\documentclass[11pt]{amsart}

\let\oldmarginpar\marginpar
\renewcommand\marginpar[1]{\-\oldmarginpar[\raggedleft\footnotesize #1]%
{\raggedright\footnotesize #1}}

\usepackage{savesym}
\usepackage{newtxtext,newtxmath}
\restoresymbol{TXF}{Bbbk}

\usepackage{tikz-cd}
\usepackage{quiver}
\usepackage{amsmath}
\usepackage{amsthm}
\usepackage{xspace}
\usepackage[all,tips]{xy}
\usepackage{syntonly}

\usepackage{arydshln}
\usepackage{todo}
\usepackage{enumerate}
\usepackage{graphicx}
\usepackage{fancyhdr}
\usepackage{graphics}
\usepackage{setspace}
\usepackage[colorlinks=true,linkcolor=blue,citecolor=blue]{hyperref}

\theoremstyle{plain}
\newtheorem{thm}{Theorem}[section]

\newtheorem{prop}[thm]{Proposition}
\newtheorem{lemma}[thm]{Lemma}
\newtheorem{ques}{Question}

\theoremstyle{definition}
\newtheorem{rmk}[thm]{Remark}
\newtheorem{defn}[thm]{Definition}

\newtheorem{ex}[thm]{Example}

\newcommand{\Z}{\mathbb{Z}}
\newcommand{\N}{\mathbb{N}}
\newcommand{\Q}{\mathbb{Q}}
\newcommand{\stab}{\mbox{Stab}}

\title{Grothendieck Rigidity and virtual retraction of higher-rank GBS groups}

\author{Daxun Wang}
\address{D. Wang: Yau Mathematical Sciences Center, Tsinghua University, Beijing, China}
\email{wangdaxun@mail.tsinghua.edu.cn}



\begin{document}

\begin{abstract}
A rank $n$ generalized Baumslag-Solitar group ($GBS_n$ group) is a group that splits as a finite graph of groups such that all vertex and edge groups are isomorphic to $\Z^n$. This paper investigates Grothendieck rigidity and virtual retraction properties of $GBS_n$ groups. We show that every residually finite $GBS_n$ group is Grothendieck rigid. Furthermore, we characterize when a $GBS_n$ group satisfies property (VRC), showing that it holds precisely when the monodromy is finite.

\end{abstract}

\maketitle

\section{Introduction}
Let $u:H \rightarrow  G$ be a homomorphism  between two finitely generated residually finite groups. Then $u$ induces a homomorphism $\hat{u} : \widehat{H}\rightarrow \widehat{G}$ between the profinite completions. In \cite{Gro70}, Grothendieck asked the following question.

\begin{ques}\label{ques: question 1}
Let $u: H\hookrightarrow G$ be an embedding between two finitely presented residually finite groups such that the induced homomorphism $\hat{u}:\widehat{H}\rightarrow \widehat{G}$ is an isomorphism. Does it follow that $u$ is an isomorphism?
\end{ques}

We assume that all groups are finitely generated and residually finite unless otherwise stated.
We say that $(G,H)$ is a \textit{Grothendieck pair} if $H$ is
a proper subgroup of $G$ and the induced homomorphism $\hat{u}$ is an isomorphism. Moreover, we say that $G$ is \textit{Grothendieck rigid} if for any proper subgroup $H$ of $G$, $(G,H)$ is not a Grothendieck pair. 

In \cite{PT90}, Platonov and Tavgen' constructed the first example of Grothendieck pairs $(G,H)$ consisting of finitely generated (but infinitely presented) residually finite groups. Later, Bridson and Grunewald gave an
example where both groups were finitely presented in \cite{BG04}. 

When restricted to 3-manifold groups, Sun showed that all finitely generated 3-manifold groups are Grothendieck rigid \cite{Sun23}. The proof relies on analyzing subgroup separability and graph of groups structures of $G$ and its subgroup $H$. For non-separable subgroups $H$, there exists a special free subgroup $H_0$ of a vertex group, whose normalizer has finite index in $H$ but infinite index $G$. This index mismatch persists in the profinite completions, preventing an isomorphism between them. 

In this paper, we also restrict our attention to the class of finitely generated groups that are equipped with a graph of groups structure. In particular, we consider the class of groups that split as a finite graph of groups such that all vertex and edges groups are isomorphic to $\Z^n$. These groups are also called \textit{rank $n$ generalized Baumslag-Solitar groups ($GBS_n$ groups)}. We remark that, since the underlying graphs of these groups are finite, each $GBS_n$ group is finitely generated and hence finitely presented. $GBS_n$ groups have been studied in many aspects such as hierarchically hyperbolicity \cite{J25}, subgroup separability \cite{GBSn} and $C^{\ast}$-simplicity \cite{XWY25}. It is known that a $GBS_n$ group is residually finite if and only if it is either a strictly ascending HNN-extension of $\Z^n$ or virtually $\Z^n$-by-free (see \cite[Theorem 3.1]{GBSn}). We note that, for all ascending HNN-extensions of finitely generated free groups, Jaikin-Zapirain and Lubotzky showed that they are Grothendieck rigid \cite{ZL25}. In this paper, we showed that every strictly ascending HNN-extension of $\Z^n$ is Grothendieck rigid (see Proposition \ref{prop: HNN GR}). As an application, we have the following result.

\begin{thm}\label{thm: Main theorem 1}
Every residually finite $GBS_n$ group is Grothendieck rigid.
\end{thm}

Next, we study property (VRC) for  $GBS_n$ groups. Recall that if $G$ is a group, a subgroup $H$ is a \textit{virtual retract} of $G$, denoted $H\leq_{vr} G$, if there is a finite index subgroup $K$ of $G$ with $H\leq K$ and a homomorphism $\phi: K \rightarrow H$ which is the identity map when restricted to $H$. A group $G$ has \textit{property (VRC)}  if every cyclic subgroup is a virtual retract of $G$. Property (VRC) is very useful for studying the profinite topology on groups. It is known that a virtual retract of a residually finite group is closed in the profinite topology (see \cite[Lemma 2.2]{Min21}). Therefore, property (VRC) implies that the group is cyclic
subgroup separable. Property (VRC) have been studied in particular in the class of groups that are equipped with a graph of free abelian groups structure (see \cite{Min21}, \cite{UM251},\cite{UM252},\cite{WY25}). As for $GBS_n$ groups, we have the following result.

\begin{thm}\label{thm: Main theorem 2}
Let $G$ be a $GBS_n$ group. Then $G$ has (VRC) if and only if it has finite monodromy.
\end{thm}

\section{Grothendieck rigidity}

\subsection{Profinite completion}

For a group $G$, let $\mathcal{N}$ be the set of all normal subgroups with finite index in $G$. Then $\mathcal{N}$ forms an inverse system ordered by inclusion. We define the \textit{profinite completion $\widehat{G}$} of $G$ as the inverse limit $$\widehat{G}=\mathop{\lim_{\longleftarrow}}_{N\in \mathcal{N}} G/N$$
There is a natural homomorphism $\iota_G: G\rightarrow \widehat{G}$ given by $g\mapsto (gN)_{N\in \mathcal{N}}$. This homomorphism is not injective in general. In fact, $\iota_G$ is injective if and only if $G$ is residually finite. By \cite[Corollary 1.2.28]{Wilkes24}, the map $\iota_G$ has dense image in $\widehat{G}$.

\subsection{Higher-rank GBS groups}
In this subsection, we study Question \ref{ques: question 1} for $GBS_n$ groups. First, we show that a strictly ascending HNN-extension of $\Z^n$ is Grothendieck rigid. Next, as an application, we derive that every residually finite $GBS_n$ group is Grothendieck rigid.

Let $G$ be a group that has the following presentation: 
$$G=\langle A, t\ |\ tat^{-1}= \phi(a), \forall a\in A \rangle$$ 
where $A=\Z^n$, and $\phi:\Z^n\rightarrow \Z^n$ is given by left multiplication by a matrix $\mathscr{M}\in M_{n\times n}(\Z)$ such that $|\mbox{det} \mathscr{M}|>1$. The group $G$ is residually finite by \cite[Theorem 3.1]{GBSn}. For each $l\in \N$, we define the subgroup $A_l=\phi^{-l}(A)=t^{-l}At^l$. This forms an ascending chain $A_0\subset A_1\subset A_2 \subset ...$ and we denote the direct limit of this ascending chain by $N=\varinjlim A_l= \bigcup_{l\in \N}\phi^{-l}(A)$. In fact, $N$ is the normal closure of $A$ in $G$. Since $A$ is abelian, we have $N$ is also abelian.
Further, we have $G=N\rtimes \langle t\rangle$.

Let $H$ be a finitely generated subgroup of $G$ such that $(G,H)$ is a Grothendieck pair. The following lemma tells us the structure of $H$.

\begin{lemma}\label{lemma: H semidirectproduct}
Let $H$ be a finitely generated subgroup of $G$ such that the natural inclusion $i: H \hookrightarrow G$ induces an isomorphism $\widehat{i}: \widehat{H} \rightarrow \widehat{G}$. Then the following holds:
\begin{enumerate}
    \item $H= K \rtimes \langle t_H\rangle$ where $K=H\cap N$, and $t_H$ is an infinite order element in $H$ that acts on $K$ by conjugation. In particular, we have $t_Hat_{H}^{-1}=\phi(a)$ for any $a\in K\cap A$.
    \item $K=\bigcup_{l\in \N}t^{-l}(A\cap K)t^l$.
\end{enumerate}
\end{lemma}

\begin{proof}
Let $\pi:G\rightarrow \Z$ be the projection map given by $\pi(N)=0$ and $\pi(t)=1$. Since $N$ is normal in $G$, then we have $\mbox{ker}(\pi)=N$. Under this projection map we have $\pi(H)=m\Z$ for some $m\geq 0$. It follows that $\pi^{-1}(m\Z)=\langle N, t^m\rangle$ is a finite index subgroup of $G$. By \cite[Theorem 1.3.20]{Wilkes24} we have $[\widehat{G}: \overline{\pi^{-1}(m\Z)}]=[G:\pi^{-1}(m\Z)]=m$. We claim that $m=1$. Assume, by contradiction, that $m\neq 1$. Then there are two cases:
\begin{enumerate}
    \item[(i)]Case $m=0$. In this case, we have $H\subseteq N$. Let $\widehat{\pi}: \widehat{G}\rightarrow \widehat{\Z}$ be the unique continuous map induced by $\pi$. Then we have the following commutative diagram:
    \[
    \begin{tikzcd}
& G \arrow[d, "\iota_G"] \arrow[r, "\pi"] & \Z \arrow[d, "\iota_{\Z}"]  \\
& \widehat{G} \ar[r, "\widehat{\pi}"] & \widehat{\Z}
\end{tikzcd}
\]
where $\iota_G,\iota_{\Z}$ are the natural dense embeddings. Since $\pi$ is surjective, the map $\iota_{\Z}\circ \pi=\widehat{\pi}\circ \iota_G$ has dense image in $\widehat{\Z}$. Therefore the image of $\widehat{\pi}$ is dense in $\widehat{\Z}$. It is also compact as $\widehat{\pi}$ is continuous, hence the image of $\widehat{\pi}$ is all of $\widehat{\Z}$, i.e. $\widehat{\pi}$ is surjective. Note that $\overline{N}=\mbox{ker}(\widehat{\pi})$. Since $N$ is a proper subgroup of $G$ and $\overline{N}$ is a proper closed subgroup of $\widehat{G}$. The isomorphism $\widehat{i}: \widehat{H}\cong \widehat{G}$ forces $\overline{H}=\widehat{G}$, so $H$ cannot be contained in the proper closed subgroup $\overline{N}$. Therefore this case cannot happen.
\item[(ii)] Case $m>1$. In this case, we have $H\subseteq \pi^{-1}(m\Z)$. Since $\overline{\pi^{-1}(m\Z)}$ has finite index in $\widehat{G}$, it is a proper open subgroup of $\widehat{G}$. Then $\overline{H}\subseteq \overline{\pi^{-1}(m\Z)}$ is a proper closed subgroup of $\widehat{G}$. Since $\widehat{H}\cong \widehat{G}$ forces $\overline{H}=\widehat{G}$, so $H$ cannot be contained in the proper closed subgroup $\overline{\pi^{-1}(m\Z)}$. Therefore this case cannot happen.
\end{enumerate}
By the above argument, we have $\pi(H)=\Z$. Thus there exists $t_H\in H$ such that $\pi(t_H)=1$. By Britton's lemma, we can write $t_H=tn$ for some $n\in N$. 

Next, we show that $H$ can be written as a semidirect product $H= K \rtimes \langle t_H\rangle$ where $K=H\cap N$. Since $N \lhd G$, we have $K \lhd H$. Since $K=\mbox{ker}(\pi|_H)$ and $t_H\not\in \mbox{ker}(\pi|_H)$, we have $K\cap \langle t_H\rangle=1$. For any element $h\in H$, we have $\pi(h)=r$ for some $r\in \Z$. Define $k=t_{H}^{-r}h$. Then $\pi(k)=\pi(h)-r=0$. This implies that $k\in K$ and $h=t_{H}^rk$. Therefore, we have $H=K\cdot \langle t_H\rangle$. Thus, we have $H= K \rtimes \langle t_H\rangle$. In particular, since $N$ is abelian, we have $t_H a t_H^{-1}=(tn) a (n^{-1}t^{-1})=tat^{-1}=\phi(a)$ for any $a\in K\cap A$. Thus condition (i) holds.

From the above argument, we note that $K$ is normalized by $t$, i.e. $K=t_HKt_H^{-1}=(tn)K(tn)^{-1}=tKt^{-1}$. Now let $k\in K\subseteq N$, we can write $k=t^{-l}at^l$ for some $l\in \N$ and some $a\in A$. Then $a=t^lkt^{-l} \in K$. This implies that $a\in A\cap K$. 
Therefore, we have $K\subseteq \bigcup_{l\in \N}t^{-l}(A\cap K)t^l$. Since $t^{-1}Kt=K$, then we have $K=\bigcup_{l\in \N}t^{-l}(A\cap K)t^l$, and thus condition (2) holds.
\end{proof}

We note that if $(G,H)$ is a Grothendieck pair, then the sets of isomorphism types of finite quotients of $G$ and $H$ are equal (see \cite[Corollary 3.2.4]{Wilkes24}). In the following lemma, we show that if $(G,H)$ is a Grothendieck pair, then every finite quotient map of $G$ restricts to a finite quotient map of $H$.

\begin{lemma}\label{lemma: restriction is surjective}
Let $H$ be a finitely generated subgroup of $G$ such that the natural inclusion $i: H \hookrightarrow G$ induces an isomorphism $\widehat{i}: \widehat{H} \rightarrow \widehat{G}$. Then for every finite quotient map $\pi:G\rightarrow F$, the restriction $\pi|_H: H\rightarrow F$ is surjective.  
\end{lemma}
\begin{proof}
For any finitely generated subgroup $H$ of $G$, we have the following commutative diagram:
\[
\begin{tikzcd}
& H \arrow[d, "\iota_H"] \arrow[r, "i"] & G \arrow[d, "\iota_G"] \arrow [r, "\pi"]&  F \\
& \widehat{H} \arrow[r, "\widehat{i}"] & \widehat{G} \arrow[ru,dashed, "\widehat{\pi}"{xshift=10pt, yshift=-15pt}] &
\end{tikzcd}
\]
where $\iota_G,\iota_H$ are the natural dense embeddings. Consider the finite quotient map $\pi:G\rightarrow F$. By the universal property of profinite completion, $\pi$ induces a continuous map $\widehat{\pi}:\widehat{G}\rightarrow F$ such that $\pi=\widehat{\pi}\circ \iota_G$. Since $\pi$ is surjective, then $\widehat{\pi}$ is also surjective. This implies that the composition map $\widehat{\pi}\circ \widehat{i}: \widehat{H}\rightarrow F$ is also continuous and surjective. Since $\iota_H(H)$ is dense in $\widehat{H}$, and the image $(\widehat{\pi}\circ \widehat{i})(\iota_H(H))$ is dense in $F$. But $F$ is finite, so the only dense subset is the whole set. Therefore, $(\widehat{\pi}\circ \widehat{i})(\iota_H(H))=F$. Since $(\widehat{\pi}\circ \widehat{i})\circ \iota_H=\widehat{\pi}\circ \iota_G \circ i=\pi\circ i=\pi|_H$, then we have $\pi|_H(H)=F$. Thus, the restriction $\pi|_H$ is surjective.
\end{proof}

In the following lemma, we construct a specific finite quotient group of $G$. By studying images of $G$ and $H$ in this specific finite quotient, we are able to conclude that the base group $K$ of $H$ contains $A$.

\begin{lemma}\label{lemma: Q trivial}
If $(G,H)$ is a Grothendieck pair, then $|A: A\cap K|<\infty$.
\end{lemma}
\begin{proof} First, we construct a specific finite quotient of $G$ as follows. Recall that $G=A\rtimes_{\phi}\langle t\rangle$ where $A=\Z^n$ and $\phi:A\rightarrow A$ is an injective homomorphism given by left multiplication by a matrix $\mathscr{M}\in M_{n\times n}(\Z)$.
Let $p$ be a prime number with $p\nmid \mbox{det}(\mathscr{M})$. The map $\phi$ induces an automorphism $\bar{\phi}:A/pA\rightarrow A/pA$ given by left multiplication by the matrix $\overline{\mathscr{M}}\in \mbox{GL}(n,\Z/p\Z)$, where the $(i,j)$-entry of $\overline{\mathscr{M}}$ is the $(i,j)$-entry of $\mathscr{M}$ modulo $p$. Since $\mbox{GL}(n,\Z/p\Z)$ is finite, $\overline{\mathscr{M}}$ has finite order. We denote by $r$ the order of $\overline{\mathscr{M}}$, which is also the order of $\bar{\phi}$. Choose $m=pr$. Now we define the finite quotient group $F=(A/pA)\rtimes_{\bar{\phi}}(\Z/m\Z)$, and define a surjective homomorphism $\pi: G\rightarrow F$ given by $\pi(a)= (a \ \mbox{mod}\ p, 0)$ and $\pi(t)=(0,1)$. Since $r\mid m$, we have $\bar\phi^m=\mathrm{id}_{A/pA}$, so the action of $\Z/m\Z$ on $A/pA$ is well-defined. Moreover, since $\pi(t a t^{-1}) = (0,1)(a,0)(0,-1) = (\bar{\phi}(a), 0) = \pi(\phi(a))$, $\pi$ is a well-defined homomorphism.

Next, we study the images of $G$ and $H$ in the finite quotient group $F$. Since $(G,H)$ is a Grothendieck pair, by Lemma \ref{lemma: restriction is surjective}, we have $\pi(H)=F$. Recall that $H=\langle K, tn\rangle$. Then we have $F=\langle \pi(K), \pi(tn)\rangle$. Since $K$ is normalized by $t$ in $G$ and $K=\bigcup_{l\in \N}t^{-l}(A\cap K)t^l$, we have $\pi(K)=\bigcup_{l\in \N}\pi(t)^{-l}\pi(A\cap K)\pi(t)^l=\bigcup_{l\in \N}\bar{\phi}^{-l}(\pi(A\cap K))$. Notice that $t (A \cap K) t^{-1} \subseteq t A t^{-1} \cap t K t^{-1} \subseteq A \cap K$. This implies that $\bar{\phi}(\pi (A\cap K))=\pi(\phi (A\cap K))\subseteq \pi(A\cap K)$. Since $\pi(A \cap K)$ is a finite set and $\bar{\phi}$ is injective, its restriction to $\pi(A \cap K)$ is a bijection from $\pi(A \cap K)$ onto itself. This means that $\pi(K) = \pi(A \cap K)$. We denote this subgroup by $S\subseteq A/pA$.

Since $F=\langle S, \pi(tn)\rangle$ and $S$ is normal in $F$, the quotient $F/S$ is generated by the single element $\pi(tn)S$. In other words, $F/S$ is a cyclic group. However, $F/S\cong ((A/pA)/S)\rtimes (\Z/m\Z)$ because $S$ is invariant under the action of $\bar{\phi}$. Since $F/S$ is cyclic, it must be abelian. Therefore, $F/S\cong ((A/pA)/S)\times (\Z/m\Z)$.

Now we claim that $(A/pA)/S)$ must be trivial. Assume, for contradiction, that $(A/pA)/S\neq0$. Notice that every element in the quotient $(A/pA)/S\neq0$ must have order 1 or $p$. Since $(A/pA)/S\neq0$, it contains a subgroup isomorphic to $\Z/p\Z$. Recall that $m=pr$. Then the cyclic group $F/S$ contains a subgroup isomorphic to $(\Z/p\Z) \times (\Z/p\Z)$, which is impossible.

Since $(A/pA)/S=0$, then we have $\pi(A\cap K)=A/pA$. This implies that $A=(A\cap K)+pA$. Let $Q = A / (A \cap K)$. We have shown that $Q = pQ$. Since $Q$ is a finitely generated abelian group, then $Q \cong \mathbb{Z}^q \oplus T$. The equation $Q = pQ$ shows that its free rank $q$ must be 0. Thus $Q \cong T$ is finite, proving $|A : A \cap K| <\infty$.

\end{proof}

\begin{prop}\label{prop: HNN GR}
Let $G$ be a strictly ascending HNN-extension of $\Z^n$. Then $G$ is Grothendieck rigid.
\end{prop}

\begin{proof}
Let $H$ be a finitely generated subgroup of $G$ such that $(G,H)$ be a Grothendieck pair. By Lemma \ref{lemma: H semidirectproduct}, we have $H=K\rtimes \langle t_H\rangle$. Further, by Lemma \ref{lemma: Q trivial}, we know that $|A:A\cap K|<\infty$. Let $Q=A/(A\cap K)$ be the finite quotient group. Since $A\cap K$ is $\phi$-invariant, then the map $\phi$ induces a well-defined map on the quotient group $\bar{\phi}: Q \rightarrow Q$ given by $a+(A\cap K)\mapsto \phi(a)+(A\cap K)$. Notice that, for any $a\in A$ and $\phi(a)\in A\cap K$, we have $tat^{-1}\in K$. Since $K$ is normalized by $t$ in $G$, we have $a\in t^{-1}Kt=K$. Therefore $a\in A\cap K$. This shows that $\mbox{Ker}(\bar{\phi})$ is trivial and hence $\bar{\phi}$ is injective. Since $Q$ is finite, $\bar{\phi}$ is an automorphism of $Q$. We denote by $r$ the finite order of $\bar{\phi}$.

For any integer $s\geq 1$, let $m=sr$. We can define a finite group $Q \rtimes_{\bar{\phi}} \Z/m\Z$ and the finite quotient map $\psi_s : G \to Q \rtimes_{\bar{\phi}} \Z/m\Z$ given by $\psi_s(a) = (\bar{a}, 0)$ and $\psi_s(t) = (0, 1)$. One can check that $\psi_s$ is a well-defined homomorphism. By Lemma \ref{lemma: restriction is surjective}, $\psi_s\vert{}_H$ must be surjective. Since $\psi_s(A\cap K)=(0,0)$, by (2) of Lemma \ref{lemma: H semidirectproduct}, we have $\psi_s(K)=(0,0)$. Recall that $H = \langle K, tn \rangle$. Then the image $\psi_s(H)$ is generated by the single element $\psi_s(tn)$. Therefore, for any $s\geq 1$, the target group $Q \rtimes_{\bar{\phi}} \Z/sr\Z$ is a cyclic group. Since it is cyclic, it must be abelian, meaning that $\bar{\phi}$ acts trivially on $Q$. Therefore, the order $r = 1$ and the group is the direct product $Q \times \Z/s\Z$.

Now we claim that $Q$ must be trivial. Assume, for contradiction, that $Q$ is not trivial. Let $q$ be a prime dividing $|Q|$ and let $s = q$. The finite quotient group is $Q \times \mathbb{Z}/q\mathbb{Z}$. Since $q$ divides $|Q|$, $Q$ has a subgroup isomorphic to $\Z/q\Z$. This means the cyclic group $Q \times \mathbb{Z}/q\mathbb{Z}$ has a subgroup isomorphic to $\Z/q\Z\times \Z/q\Z$, which is impossible.

Since $Q$ is trivial, we have $A=A\cap K$, meaning that $A\subseteq K$. Since $t$ normalizes $K$ and $A \subseteq K$, we have $N = \bigcup t^{-l} A t^l \subseteq K$. Since $K = H \cap N \subseteq N$, this forces that $K = N$. Recall $t_H = t n \in H$. Because $n \in N = K \subseteq H$, we have $t = t_H n^{-1} \in H$. Since $H$ contains both $N$ and $t$, we have $H = \langle N, t \rangle = G$.

\end{proof}

\begin{proof}[Proof of Theorem \ref{thm: Main theorem 1}]
Let $G$ be a residually finite $GBS_n$ group. By \cite[Theorem 3.1, 4.2]{GBSn}, $G$ is either a strictly ascending HNN-extension of $\Z^n$ or LERF. Then Grothendieck rigidity of $G$ follows directly from Proposition $\ref{prop: HNN GR}$ and \cite[Corollary 2.3]{Sun23}.
\end{proof}

\section{Property (VRC)}
In this section, we study property (VRC) for $GBS_n$ groups. First, we provide some basic notions about $GBS_n$ groups, more details can be found in \cite{J25}.

\begin{defn}\label{defn: modular homomorphism}
Let $G$ be a $GBS_n$ group for $n\geq 1$. Fix a base vertex $v_0$ in the tree $T$ on which $G$ acts, a finite index subgroup $A\cong \Z^n$ of $\stab_G(v_0)\cong \Z^n$ and an ordered basis $a_1,...,a_n$ of $A$. The \textit{modular homomorphism} $\mathcal{M}:G\rightarrow GL(n,\Q)$ is defined as follows: for any $g\in G$, the subgroups $A$ and $g^{-1}Ag$ are commensurable. There is $m\in \mathbb{N}$ such that
$$a_j^m\in A\cap g^{-1}Ag \ \ \mbox{for all}\ j=1,...,n$$
Then $ga_j^mg^{-1}\in A$, so we may write 
$$ga_j^mg^{-1}=a_1^{g_{1j}}\cdots a_n^{g_{nj}}$$
and define $\mathcal{M}(g)$ to be the matrix whose $(i,j)$-entry is $g_{ij}/m\in \Q$. This is independent of the choice of $m$, and $\mathcal{M}$ is a group homomorphism. The \textit{monodromy} of a $GBS_n$ group $G$ is the image $\mathcal{M}(G)$.
\end{defn}

\begin{rmk}
We note that changing the choices of the base vertex $v_0$, the base $\Z^n$-subgroup and the ordered basis of $A$ result in a modular homomorphism $\mathcal{M'}$ that is conjugate to $\mathcal{M}$ in $GL(n,\Q)$. However, the property of $\mathcal{M}(G)$ being finite is invariant under conjugation. Therefore, all results concerning finite monodromy is independent of the choices made in defining $\mathcal{M}$.
\end{rmk}

\begin{proof}[Proof of Theorem \ref{thm: Main theorem 2}]
Assume that $G$ has (VRC). We fix a base vertex $v_0$ in the tree $T$. Let $A=\mbox{Stab}_G(v_0)$ be the base $\Z^n$-subgroup and let $a_1,...,a_n$ be an ordered basis of $A$. Denote $N=\cap_{v\in V(T)}\mbox{Stab}_G(v)$. We first claim that $N$ is a normal subgroup in $G$ that is isomorphic to $\Z^n$. Since $G$ has (VRC), by \cite[Lemma 2.3]{Min21} $G$ is residually finite. It follows from \cite[Theorem 3.1, 4.2]{GBSn} that $G$ is either a strict ascending HNN extension of $\Z^n$ or LERF. However, a strict ascending HNN extension of $\Z^n$ cannot have (VRC), see \cite[Remark 9.3 (a)]{Min21}. Thus, $G$ is LERF. Since $G$ is the fundamental group of
a finite graph of groups where all of the finitely many edge and vertex groups are commensurable, by \cite[Theorem 2]{RV96}, we have $\cap_{g\in G}gG_xg^{-1}$ has finite index in $G_x$ for any vertex group $G_x$. Therefore $N=\cap_x(\cap_{g\in G} gG_xg^{-1})$ has finite index in $G_x$ and $N \cong \Z^n$. The normality of $N$ follows from the fact that the action of $G$ on $T$ induces an homomorphism $f:G\rightarrow \mbox{Aut}(T)$ and $N=\mbox{ker}(f)$. 

Next, since $N\lhd G$, then the action of $G$ on $N$ by conjugation gives rise to a homomorphism $\psi:G\rightarrow \mbox{Out}(N)$.  By \cite[Lemma 5.1, Proposition 1.5]{Min21}, we have $N$ is a virtual retract of $G$. Moreover, by \cite[Lemma 4.5]{UM251}, we have $\psi(G)$ is finite. This implies that there are only finitely many automorphisms $\alpha_g : N \rightarrow N$ arising as conjugation by $g\in G$. Since $N$ has finite index in $A$, then there exists $m\in \N$ such that $a_j^m\in N$. Further, since $N\lhd G$, then $ga_j^mg^{-1}=\alpha_g(a_j^m)\in N$. Since there are only finitely many $\alpha_g$, each $g_{ij}$ has only finitely many choices. Therefore the rational matrices $\mathcal{M}(g)$ range over a finite subset of $GL(n,\Q)$. Thus we have $\mathcal{M}(G)$ is finite.

Conversely, suppose that $G$ has finite monodromy. It follows from \cite[Theorem 3.4]{J25} that $G$ is virtually $\Z^n\times F_r$ where $F_r$ is a free group of rank $r\geq 0$. Since $\Z^n\times F_r$ has (VRC) and (VRC) is stable under passing to finite index supergroups (\cite[Lemma 5.2]{Min21}), then $G$ has (VRC).
\end{proof}

\begin{ex}
The Leary-Minasyan group given by the finite presentation
$$L=\langle t,a,b \ |\ [a,b]=1, ta^2b^{-1}t^{-1}=a^2b,\ tab^2t^{-1}=a^{-1}b^2\rangle$$
is a $GBS_2$ group that does not have (VRC) because it has infinite monodromy (\cite[Corollary 5.4]{J25}).
\end{ex}

\subsection*{Acknowledgements}
The author is supported by the National Natural Science Foundation of China (Grant No. 12501027). The author would like to thank Ashot Minasyan for many helpful discussions on property (VRC) for $GBS_n$ groups.

\bibliographystyle{alpha}
\bibliography{citations}

\end{document}